\documentclass{amsart}

\usepackage{amssymb}
\usepackage{amstext}
\usepackage{amsmath}
\usepackage{amscd}
\usepackage{amsthm}
\usepackage{amsfonts}
\usepackage{enumerate}
\usepackage{graphicx}
\usepackage{color}
\usepackage{here} % 図や表を強制的に出力させる
\usepackage{bm} % 太字ベクトル
\usepackage{mathrsfs}
\usepackage{mathtools}
\usepackage{latexsym}
\usepackage{booktabs}
\usepackage{fancyhdr}
\usepackage{subcaption}
\usepackage{tikz}
\usepackage{tikz-cd}
\usetikzlibrary{arrows}
\usetikzlibrary{decorations.markings}
\usetikzlibrary{patterns,decorations.pathreplacing}
\usepackage{mathtools}
\usepackage{pdfpages}
\usepackage[section]{algorithm}
\usepackage{algpseudocodex}

\usepackage[noadjust]{cite}
\usepackage{hyperref}
\hypersetup{
  colorlinks=true,
  citecolor=red,
  linkcolor=blue
}
\theoremstyle{plain}
\newtheorem{theorem}{Theorem}[section]
\newtheorem{lemma}[theorem]{Lemma}
\newtheorem{proposition}[theorem]{Proposition}
\newtheorem{corollary}[theorem]{Corollary}

\theoremstyle{definition}

\newtheorem{definition}[theorem]{Definition}

\theoremstyle{remark}

\newcommand{\groebner}{Gr\"{o}bner }

\newcommand{\intnonneg}{\mathbb{Z}_{\geq 0}}

\newcommand{\calB}{\mathcal{B}}

\newcommand{\calG}{\mathcal{G}}

\newcommand{\calU}{\mathcal{U}}

\newcommand{\ZZ}{\mathbb{Z}}

\newcommand{\kk}{\Bbbk}

\newcommand{\ab}{\mathbf{a}}

\newcommand{\eb}{\mathbf{e}}

\newcommand{\tb}{\mathbf{t}}

\DeclareMathOperator{\diag}{diag}

\DeclareMathOperator{\GL}{GL}
\DeclareMathOperator{\borel}{B}

\DeclareMathOperator{\ini}{in}

\DeclareMathOperator{\gin}{gin}
\DeclareMathOperator{\HS}{HS}
\DeclareMathOperator{\HF}{HF}

\title{Cartwright--Sturmfelsness of complementary determinantal edge ideals}
\date{}

\author[K. Tani]{Koichiro Tani}
\address{Department of Pure and Applied Mathematics, Graduate School of Information
Science and Technology, The University of Osaka, Osaka, Japan}
\email{tani-k@ist.osaka-u.ac.jp}


\newcommand{\one}{\mathbf{1}}

\begin{document}
\begin{abstract}
We introduce complementary determinantal edge ideals. For a graph $G$ on $n$ vertices,
the complementary determinantal edge ideal $J_{c}(G)$ is generated by the maximal minors
of a generic $(n-2)\times n$ matrix obtained by deleting the two columns
indexed by each edge of $G$. We completely characterize the graphs
for which $J_{c}(G)$ is Cartwright--Sturmfels with respect to the grading
by columns. We prove that this property is equivalent to the existence
of a multilinear universal \groebner basis and characterize
the graphs satisfying these equivalent conditions.
In particular, the ideals satisfying these equivalent conditions are radical.
\end{abstract}
\maketitle

\section{Introduction}\label{section:Introduction}
Many classes of ideals arising from graphs have been studied, including edge ideals~\cite{VillarrealEdgeIdeal},
complementary edge ideals~\cite{FicarraMoradiComplementaryEdgeIdeals, HibiComplementaryEdgeIdeals},
binomial edge ideals~\cite{HerzogHibiBinomialEdgeIdeal,OhtaniBinomialEdgeIdeal},
and generalized binomial edge ideals~\cite{GeneralizedBinomialEdgeIdeals}.
Generators of these ideals are of squarefree monomials or binomials associated with edges of graphs.
In particular, binomial edge ideals and generalized binomial edge ideals are generated by
$2$-minors of matrices of indeterminates associated with edges of graphs.
\groebner bases of these ideals are computed in~\cite{HerzogHibiBinomialEdgeIdeal,OhtaniBinomialEdgeIdeal,GeneralizedBinomialEdgeIdeals},
and they have squarefree initial ideals, so they are radical.
As a generalization of determinantal ideals of maximal minors,
there exists a class of determinantal facet ideals~\cite{EneHerzogHibiMohammadi}.
Determinantal facet ideals have not yet been studied extensively,
and even the structure of their \groebner bases and the conditions
under which they are radical are not well understood.
Therefore, in this paper, we introduce a new class of ideals generated by maximal minors
associated with edges of graphs.

Let $G$ be a simple graph on the vertex set $[n] = \{1, \ldots, n\}$. We denote the edge set of $G$ by $E(G)$.
In terms of edge ideals, each edge $\{i,j\} \in E(G)$ is associated with the product
of two variables $x_i x_j$. In terms of complementary edge ideals, each edge $\{i,j\}$
is associated with the product of all variables except two variables $x_1\cdots x_n / x_i x_j$.
In terms of binomial edge ideals, each edge $\{i,j\}$ is associated with the determinant of the submatrix
$X_{[2], \{i,j\}}$ of a $2 \times n$ generic matrix $X$ consisting of two columns.
The class of complementary determinantal edge ideals can be viewed as a complementary
analogue of the class of binomial edge ideals. We associate each edge $\{i,j\}$
with the determinant of the submatrix consisting of all columns except two columns.
Let $X = (x_{ij})$ be a $(n-2) \times n$-matrix of indeterminates and $S = \kk[x_{ij} : 1 \leq i \leq n-2, 1 \leq j \leq n]$
be a polynomial ring over a field $\kk$.
Let $\mu_{ij} = \det X_{[n-2], [n] \setminus \{i,j\}}$.
The \textit{complementary determinantal edge ideal} of $G$ is
\begin{equation*}
 J_{c}(G)=(\mu_{ij}:ij\in E(G))\subset S.
\end{equation*}
If $G$ is the complete graph $K_n$, the complementary determinantal edge ideal $J_c(K_n)$ coincides with the determinantal ideal $I_{n-2}(X)$ of
maximal minors of $X$. Therefore, complementary determinantal edge ideals are generalizations of determinantal ideals generated by
maximal minors of $(n-2) \times n$ generic matrices.

The class of complementary determinantal edge ideals is a subclass of determinantal facet ideals.
For a graph $G$, we associate the following simplicial complex
\[
 \Delta_G=\langle F_{ij}:ij\in E(G)\rangle.
\]
If $E(G)\neq\varnothing$, it is pure of dimension $n-3$, and $J_{c}(G)$ is a determinantal facet ideal associated with $\Delta_G$.
If $n = 3$, then $J_{c}(G)$ is an ideal generated by variables. If $n = 4$, then $J_{c}(G)$ is a binomial
edge ideal of a graph with four vertices.

In this paper, we investigate the Cartwright--Sturmfels property of complementary determinantal edge ideals.
Let $n$ and $m_1, \ldots, m_n$ be positive integers.
Let $S = \kk[x_{i,j} : 1 \leq j \leq n, 1 \leq i \leq m_j]$ be a standard $\ZZ^n$-graded
polynomial ring over a field $\kk$.
We use a multigrading $\deg(x_{ij})=\eb_j\in\ZZ^n$, where
$\eb_1,\ldots,\eb_n$ are the standard basis vectors, and write $\one=(1,\ldots,1)$.
Let $\calG = \GL_{m_1}(\kk) \times \cdots \times \GL_{m_n}(\kk)$.
We assume that $\calG$ acts on $S$.
Precisely, an element $g = (g^{(1)}, \ldots, g^{(n)}) \in \calG$ acts on the variables in $S$ by
\begin{equation*}
  g \cdot x_{i,j} = \sum_{k=1}^{m_j}g_{k,i}^{(j)} x_{k,j}.
\end{equation*}
Let $\calB = \borel_{m_1}(\kk) \times \cdots \times \borel_{m_n}(\kk)$ be the Borel subgroup
of $\calG$, where each $\borel_{m_j}(\kk)$ consists of upper triangular matrices in $\GL_{m_j}(\kk)$.
In this paper, we use the term \textit{homogeneous} in the multigraded sense.
We say that a homogeneous ideal $J \subset S$ is \textit{Borel-fixed} if $g \cdot J = J$
for all $g \in \calB$.
A homogeneous ideal $I$ is \textit{Cartwright--Sturmfels}, or \textit{CS} for short,
if there exists a radical Borel-fixed ideal $J$ which has the same multigraded Hilbert series
as $I$. If $\kk$ is infinite, $I$ is CS if and only if the multigraded generic initial ideal
of $I$ is squarefree.

The notion of CS ideals was introduced by Conca, De Negri, and Gorla~\cite{ConcaDeNegriGorlaUGBandCS}.
For many classes of ideals studied in combinatorial commutative algebra,
conditions for being CS ideals have been characterized
~\cite{ConcaDeNegriGorlaUGBforMM, ConcaDeNegriGorlaUGBandCS, ConcaDeNegriGorlaMulGinofDet,
ConcaDeNegriGorlaCSassociGraph, ConcaDeNegriGorlaRadGin, ConcaDeNegriWelker, ConcaWelker,
matsushita2025standardmultigradedhibirings}.
In particular, binomial edge ideals~\cite{ConcaDeNegriGorlaCSassociGraph} and generalized
binomial edge ideals~\cite{ConcaDeNegriWelker} are CS.
In~\cite{ConcaDeNegriWelker}, Conca, De Negri, and Welker pointed out that
a classification of when ideals generated by some maximal minors are CS
seems to be out of reach. However, we can completely classify complementary determinantal edge ideals
that are CS ideals.

A homogeneous polynomial is called \textit{multilinear} if its degree is componentwise at most $\one$.
We call a \groebner basis or a universal \groebner basis multilinear if it consists of
multilinear polynomials. It is known that if $I$ is a CS ideal, then $I$ has a multilinear
universal \groebner basis. The converse is not true in general. For more details, see
Section~\ref{section:Cartwright--Sturmfels ideals}.
In this paper, we characterize the graphs $G$ for which $J_{c}(G)$ is a CS ideal
by showing that this is equivalent to $J_{c}(G)$ having a multilinear universal \groebner basis.

\begin{theorem}[Theorem~\ref{thm:main}]
  Let $n \geq 5$ and let $G$ be a simple graph on $[n]$ with at least two edges. The following
  conditions are equivalent.
  \begin{enumerate}[(i)]
  \item $J_{c}(G)$ is CS for the column grading.
  \item $J_{c}(G)$ has a multilinear universal \groebner basis.
  \item $G$ is one of the following graph:
    \begin{enumerate}[(a)]
      \item a matching with two edges on $5$ vertices;
      \item $K_{1,n-1}$, a star with $n-1$ leaves;
      \item $K_{1,n-2}+K_1$, consisting of a star with $n-2$ leaves together with one isolated vertex;
      \item $K_{n-1}+K_1$, consisting of a complete graph on $[n]\setminus\{v\}$ together with the isolated vertex $v$;
      \item a complete graph on $[n]\setminus\{v\}$ with $v$ adjacent to exactly one other vertex;
      \item $K_n$ with the edges of a matching removed.
    \end{enumerate}
  \end{enumerate}
\end{theorem}

This paper is organized as follows. In Section~\ref{section:Cartwright--Sturmfels ideals},
we review known results of CS ideals and introduce a basic lemma about CS.
In Section~\ref{section:Cartwright--Sturmfels complementary determinantal edge ideals},
we prove our main theorem.
In Subsection~3.1, we derive necessary conditions on the vertex
degrees for $J_{c}(G)$ to have a multilinear universal \groebner basis
and classify the graphs satisfying these conditions.
In Subsection~3.2, we determine which matchings have CS
complementary determinantal edge ideals.
Finally, in Subsection~3.3, we establish the CS property for
the remaining graph families and complete the proof of the
main theorem.

\section{Cartwright--Sturmfels ideals}\label{section:Cartwright--Sturmfels ideals}
In this section, we review basics of CS ideals.
Let $S = \kk[x_{i,j} : 1 \leq j \leq n, 1 \leq i \leq m_j]$ be a standard $\ZZ^n$-graded
polynomial ring over a field $\kk$ with $\deg(x_{ij})=\eb_j\in\ZZ^n$.
For a homogeneous ideal $I$, the residue ring
$S/I = \bigoplus_{\ab \in \intnonneg^{n}} S_{\ab}/I_{\ab}$
can be considered as a $\ZZ^{n}$-graded ring.
For $\tb = (t_1, \ldots, t_n)$ and $\ab = (a_1, \ldots, a_n) \in \intnonneg^n$,
we denote $t_1^{a_1}\cdots t_n^{a_n}$ by $\tb^{\ab}$.
The \textit{multigraded Hilbert function} and the \textit{multigraded Hilbert series} of $S/I$ are defined to be
\begin{align*}
  \HF(S/I; \ab) &= \dim_{\kk}(S/I)_{\ab}, \text{ and } \\
  \HS(S/I; \tb) &= \sum_{\ab \in \intnonneg^{n}}
  \HF(S/I; \ab) \tb^{\ab}, \text{ respectively.}
\end{align*}
For a monomial order $\preceq$, we write $\ini_{\preceq}(f)$ and $\ini_{\preceq}(I)$ for
the initial monomial of $f$ and the initial ideal of $I$, respectively.

\begin{definition}[{\cite[Definition 2.4]{ConcaDeNegriGorlaUGBandCS}}]\label{definition:Cartwright--Sturmfels ideal}
  A homogeneous ideal  $I$ is \textit{Cartwright--Sturmfels}, or \textit{CS} for short,
  if there exists a squarefree
  multigraded Borel-fixed monomial ideal $J$ such that $\HS(S/I; \tb) = \HS(S/J; \tb)$.
\end{definition}

By Definition~\ref{definition:Cartwright--Sturmfels ideal}, a homogeneous ideal $I$ is
a CS ideal if and only if $\ini_{\preceq}(I)$ is a CS
ideal for a monomial order $\preceq$ because $\HS(S/I; \tb) = \HS(S/\ini_{\preceq}(I); \tb)$.
It was shown that, if $\kk$ is an infinite field, CS ideals are characterized
in terms of multigraded generic initial ideals.
In what follows, when we consider generic initial ideals, we assume $\kk$ is infinite.
As in the $\ZZ$-graded case, for a homogeneous ideal $I$ and a monomial order $\preceq$,
there exists a Zariski open dense set $U \subset \calG$ such that the initial ideal $\ini_{\preceq}(g \cdot I)$
stays the same for all $g \in U$. We denote such an initial ideal by $\gin_{\preceq}(I)$
and call it the \textit{multigraded generic initial ideal} of $I$ with respect to $\preceq$.

\begin{proposition}[{\cite[Proposition 2.6]{ConcaDeNegriGorlaUGBandCS}}]\label{proposition:generic initial ideal of Cartwright--Sturmfels ideal}
  Let $I \subset S$ be a CS ideal and let $J \subset S$ be a squarefree
  Borel-fixed monomial ideal such that $\HS(S/I; \tb) = \HS(S/J; \tb)$.
  Then, $\gin_{\preceq}(I) = J$ for every monomial order $\preceq$.
  In particular, $I$ is radical and is generated by multilinear polynomials.
\end{proposition}

As a corollary of Proposition~\ref{proposition:generic initial ideal of Cartwright--Sturmfels ideal},
if $I$ is a CS ideal, then $\ini_{\preceq}(I)$ is a squarefree monomial ideal generated by multilinear monomials
for every monomial order. Therefore, every reduced \groebner basis of $I$ consists of multilinear polynomials.
Moreover, $I$ has a universal \groebner basis consisting of multilinear polynomials.
The converse is not true in general.
For example, with $\deg x_{1,1}=\deg x_{2,1}=\eb_1$ and $\deg x_{1,2}=\deg x_{2,2}=\eb_2,$
the monomial ideal $I=(x_{1,1}x_{1,2},x_{2,1}x_{2,2})$
has the multilinear universal \groebner basis
$\{x_{1,1}x_{1,2},x_{2,1}x_{2,2}\}.$
After the grading-preserving substitutions

$$
x_{2,1}\mapsto x_{1,1}+x_{2,1}
\qquad\text{and}\qquad
x_{2,2}\mapsto x_{1,2}+x_{2,2},
$$
its reduced \groebner basis with respect to the lexicographic order induced by
$x_{1,1}>x_{2,1}>x_{1,2}>x_{2,2}$ is

$$
\{x_{1,1}x_{1,2},
x_{1,1}x_{2,2}+x_{2,1}x_{1,2}+x_{2,1}x_{2,2},
x_{2,1}x_{1,2}^2+x_{2,1}x_{1,2}x_{2,2}\}.
$$
Hence the initial ideal is
$(x_{1,1}x_{1,2},x_{1,1}x_{2,2},x_{2,1}x_{1,2}^2),$
which is not squarefree. Since the CS property is invariant under grading-preserving changes of coordinates, \(I\) is not a CS ideal.

It is known that CS ideals are closed under some natural operations.
For example, the CS property is preserved under taking quotients by homogeneous
linear forms. More precisely, if $I\subset S$ is a CS ideal and
$L\in S_{\eb_i}$ for some $i\in[n]$, then $(I+(L))/(L)$ is a CS ideal
of $S/(L)$ by~\cite[Theorem 2.16(5)]{ConcaDeNegriGorlaUGBandCS}.

CS ideals are also preserved under elimination. More precisely, let $V_i\subset S_{\eb_i}$
be a vector subspace for each $i\in[n]$, and set
$R=\kk[V_1,\ldots,V_n]$.
Then $R$ is a $\ZZ^n$-graded subring of $S$, and if $I$ is a CS ideal
of $S$, then $I\cap R$ is a CS ideal of $R$
by~\cite[Theorem 2.16(6)]{ConcaDeNegriGorlaUGBandCS}.

\begin{proposition}[{\cite[Proposition 4.3]{ConcaDeNegriWelker}}]\label{proposition:connect one column}
  Fix an integer $\ell \in [n]$. Let $I$ be a CS ideal generated by polynomials
  of degree at most $\eb_1 + \cdots + \eb_{\ell}$.
  Let $J$ be a CS ideal generated by polynomials
  of degree at most $\eb_{\ell} + \cdots + \eb_{n}$. Then $I + J$ is CS.
\end{proposition}

We use the following lemma to prove our main theorem.
\begin{lemma}\label{lem:filter}
Let $Q\subset S$ be a squarefree monomial ideal such that
$Q_\ab$ is either $0$ or $S_\ab$ for each degree
$\ab$. Then, for every homogeneous ideal $I$ and every term
order $\preceq$,
\begin{equation*}
 \ini_{\preceq}(I\cap Q)=\ini_{\preceq}(I)\cap Q.
\end{equation*}
Moreover, if $I$ is CS, then $I\cap Q$ is CS.
\end{lemma}
\begin{proof}
In each degree, $(I\cap Q)_{\ab}$ is either $I_\ab$ or $0$. Therefore, we have
\begin{equation*}
 \ini_{\preceq}(I\cap Q)=\ini_{\preceq}(I)\cap Q.
\end{equation*}
The assumption also implies that $Q$ is
invariant under every coordinate change. Taking a common
sufficiently general change of coordinates gives
\[
 \gin_{\preceq}(I\cap Q)=\gin_{\preceq}(I)\cap Q.
\]
If $I$ is CS, the right-hand side is an intersection of squarefree monomial
ideals and is squarefree.
Thus, $I \cap Q$ is CS.
\end{proof}

\section{Cartwright--Sturmfels complementary determinantal edge ideals}\label{section:Cartwright--Sturmfels complementary determinantal edge ideals}
In this section, we prove our main theorem.
Throughout this section, let $G$ be a simple graph on the vertex set $[n] = \{1, \ldots, n\}$.
We denote the edge set of $G$ by $E(G)$.
Let $X = (x_{ij})$ be a $(n-2) \times n$-matrix of indeterminates and $S = \kk[x_{ij} : 1 \leq i \leq n-2, 1 \leq j \leq n]$
be a polynomial ring over a field $\kk$.
Let $N_G(v)$ be the
neighbor sets, and set $d_G(v)=|N_G(v)|$ and
$U_v=[n]\setminus\{v\}$. A matching is a graph whose distinct edges
have disjoint endpoints; isolated vertices are allowed.
If we use a lexicographic order $\preceq$, we assume that $\preceq$ is a lexicographic order induced by
\begin{equation*}
 x_{11}>\cdots>x_{1n}>x_{21}>\cdots>x_{n-2, n}.
\end{equation*}

It is trivial that all complementary determinantal edge ideals are CS if $n=3$.
If $n=4$, complementary determinantal edge ideals are binomial edge ideals and hence are CS ideals.
In addition, it is trivial that principal ideals generated by a multilinear polynomial are CS.
Therefore, we investigate the case where $n \geq 5$ and $G$ has at least two edges.
The following is our main theorem.

\begin{theorem}\label{thm:main}
Let $n \geq 5$ and let $G$ be a simple graph on $[n]$ with at least two edges. The following
conditions are equivalent.
\begin{enumerate}[(i)]
\item $J_{c}(G)$ is CS for the column grading.
\item $J_{c}(G)$ has a multilinear universal \groebner basis.
\item $G$ is one of the following graph:
  \begin{enumerate}[(a)]
    \item a matching with two edges on $5$ vertices;
    \item $K_{1,n-1}$, a star with $n-1$ leaves;
    \item $K_{1,n-2}+K_1$, consisting of a star with $n-2$ leaves together with one isolated vertex;
    \item $K_{n-1}+K_1$, consisting of a complete graph on $[n]\setminus\{v\}$ together with the isolated vertex $v$;
    \item a complete graph on $[n]\setminus\{v\}$ with $v$ adjacent to exactly one other vertex;
    \item $K_n$ with the edges of a matching removed.
  \end{enumerate}
\end{enumerate}
\end{theorem}

\subsection{Graph-theoretic necessary conditions for multilinear universal \groebner bases}
In this subsection, we begin the proof of the implication
\textup{(ii)}$\Rightarrow$\textup{(iii)} in
Theorem~\ref{thm:main} by deriving a necessary condition
on the vertex degrees.
We first show in Lemma~\ref{lem:inheritance} that the existence
of a multilinear universal \groebner basis is preserved under
certain natural operations.
Using these operations and Lemma~\ref{lem:local-obstruction}, we prove that every vertex
of $G$ must have degree in $\{0,1,n-2,n-1\}$.
We then classify the graphs satisfying this degree condition in Lemma~\ref{lem:graph-classification}.
To complete the implication, it remains to exclude matchings
with at least two edges when $n\geq6$; this is done in
Subsection~3.2.

\begin{lemma}\label{lem:inheritance}
Let $I$ be a homogeneous ideal which has a multilinear universal \groebner basis.
Then, the following ideals also have a multilinear universal \groebner basis:
\begin{enumerate}[(a)]
\item The elimination ideal $I \cap \kk[Y]$ for a subset $Y \subset X$;
\item The ideal generated by the polynomials obtained from elements of $I$ by
setting the variables in $Y$ equal to zero for a subset $Y \subset X$;
\item The ideal generated by the polynomials obtained from elements of $I$
by setting $t=1$, where $t$ is a variable of degree $\eb_i$
for some $i\in[n]$ and no other variable has degree $\eb_i$.
\end{enumerate}
\end{lemma}

\begin{proof}
Let $\calU$ be a multilinear universal \groebner basis of $I$
consisting of homogeneous polynomials.
We show that each ideal in (a)--(c) has a multilinear
\groebner basis with respect to an arbitrary monomial order.

For (a), fix a monomial order $\preceq_Y$ on $\kk[Y]$
and a monomial order $\preceq_{X\setminus Y}$ on
$\kk[X\setminus Y]$.
Define the product order $\preceq$ on $\kk[X]$ by first
comparing the $(X\setminus Y)$-parts using
$\preceq_{X\setminus Y}$ and then the $Y$-parts using
$\preceq_Y$.
This is an elimination order for the variables $X \setminus Y$
and restricts to $\preceq_Y$ on $\kk[Y]$.
Since $\calU$ is a \groebner basis of $I$ with respect to
$\preceq$, the elimination theorem implies that
$\calU\cap\kk[Y]$ is a \groebner basis of $I\cap\kk[Y]$
with respect to $\preceq_Y$.
It is a multilinear \groebner basis.

For (b), put $Z=X\setminus Y$, and let
$\pi:\kk[X]\to\kk[Z]$ be the homomorphism obtained by setting
the variables in $Y$ equal to zero.
Fix a monomial order $\preceq_Z$ on $\kk[Z]$.
Define a monomial order $\preceq$ on $\kk[X]$ by comparing
first total degrees, then degrees in $Z$, then the $Z$-parts
using $\preceq_Z$, and finally the $Y$-parts using any fixed
monomial order on $\kk[Y]$.
This order agrees with $\preceq_Z$ on monomials of the same total degree in $\kk[Z]$.

We claim that the nonzero elements of $\pi(\calU)$ form
a \groebner basis of $\pi(I)$ with respect to $\preceq_Z$.
It suffices to consider a nonzero homogeneous polynomial
$f\in\pi(I)$.
Choose a homogeneous polynomial $F\in I$ such that $\pi(F)=f$.
Among the terms of $F$, those that survive under $\pi$
have the largest possible degree in $Z$.
Consequently,
\[
  \ini_{\preceq}(F)=\ini_{\preceq_Z}(f).
\]
Since $\calU$ is a \groebner basis of $I$ with respect to
$\preceq$, there exists $g\in\calU$ such that
\[
  \ini_{\preceq}(g)\mid\ini_{\preceq_Z}(f).
\]
In particular, $\ini_{\preceq}(g)$ involves only variables
in $Z$, so $\pi(g)\neq0$ and
\[
  \ini_{\preceq_Z}(\pi(g))=\ini_{\preceq}(g).
\]
This proves the claim.
Every nonzero $\pi(g)$ is homogeneous of the same degree
as $g$, and hence is multilinear.

For (c), put $Y=X\setminus\{t\}$, and let
$\pi:\kk[X]\to\kk[Y]$ be the homomorphism sending $t$ to $1$.
Fix a monomial order $\preceq_Y$ on $\kk[Y]$.
Define a monomial order $\preceq$ on $\kk[X]$ by first
comparing the $Y$-parts using $\preceq_Y$ and then comparing
the $t$-exponents when the $Y$-parts coincide.
Since $t$ is the only variable in its grading block,
each $g\in\calU$ has the form $g=t^{b_g}h_g$, where $b_g\ge0$ and $h_g\in\kk[Y]$.

We claim that $\{h_g:g\in\calU\}$ is a \groebner basis
of the ideal $\pi(I)$ with respect to $\preceq_Y$.
Since $\calU$ generates $I$, the polynomials $h_g=\pi(g)$
generate $\pi(I)$.
Thus, for any nonzero $f\in\pi(I)$, we may write
\[
  f=\sum_{g\in\calU}a_g h_g,
  \qquad a_g\in\kk[Y].
\]
Choose $N$ such that $N\ge b_g$ whenever $a_g\neq0$.
Then
\[
  t^Nf
  =\sum_{\substack{g\in\calU\\a_g\neq0}}
    a_g t^{N-b_g}g
  \in I.
\]
Since $\calU$ is a \groebner basis of $I$ with respect to
$\preceq$, there exists $g\in\calU$ such that
\[
  \ini_{\preceq}(g)\mid\ini_{\preceq}(t^Nf).
\]
By the definition of $\preceq$,
\[
  \ini_{\preceq}(t^Nf)
  =t^N\ini_{\preceq_Y}(f),
  \qquad
  \ini_{\preceq}(g)
  =t^{b_g}\ini_{\preceq_Y}(h_g).
\]
Since neither $h_g$ nor $f$ involves $t$, it follows that
\[
  \ini_{\preceq_Y}(h_g)\mid\ini_{\preceq_Y}(f).
\]
This proves the claim.
After forgetting the grading coordinate corresponding to $t$,
each $h_g$ remains homogeneous and multilinear.
\end{proof}

\begin{lemma}\label{lemma:3 times 4}

  Let $Y = (y_{i,j})$ be a $3 \times 4$ generic matrix with $\deg(y_{i,j}) = \eb_j$.
  Let $p_1$ and $p_2$ be
  \begin{equation*}
    p_1 = \det Y_{[3], \{2,3,4\}} \qquad p_2 = \det Y_{[3], \{1,3,4\}}.
  \end{equation*}
  Then, the ideal $I = (p_1, p_2)$ does not have any multilinear universal \groebner basis.
\end{lemma}
\begin{proof}
Put $S=\kk[y_{i,j}: 1 \leq i \leq 3,1 \leq j \leq 4]$.
Since $\deg p_1=\one-\eb_1$ and $\deg p_2=\one-\eb_2$,
the only nonzero multilinear homogeneous components of $I$ are $I_{\one-\eb_1}, I_{\one-\eb_2}$ and $I_{\one}$.

Let $\preceq$ be the lexicographic order induced by natural order of the variables.
The first $S$-polynomial is
\[
  f:=S_{\preceq}(p_2,p_1)
    =y_{1,2}p_2-y_{1,1}p_1.
\]
A direct expansion gives
\[
\begin{aligned}
  \ini_{\preceq}(p_1)
    &=y_{1,2}y_{2,3}y_{3,4},\\
  \ini_{\preceq}(p_2)
    &=y_{1,1}y_{2,3}y_{3,4},\\
  \ini_{\preceq}(f)
    &=y_{1,1}y_{1,3}y_{2,2}y_{3,4}.
\end{aligned}
\]
In what follows, the $S$-polynomial computations are carried out in degrees that are not multilinear.
Indeed, the degrees of the next $S$-polynomials $S_{\preceq}(f,p_1)$ and $S_{\preceq}(f,p_2)$ are not multilinear.
Therefore, among the minimal generators of the initial ideal $\ini_{\preceq}(I)$, the multilinear ones are
$\ini_{\preceq}(p_1)$, $\ini_{\preceq}(p_2)$, and $\ini_{\preceq}(f)$.

Set
\begin{equation*}
  h = S_{\preceq}(f,p_2) = y_{2,3}f - y_{2,2}y_{1,3}p_2.
\end{equation*}
Then, we have $\deg(h) = (1,1,2,1)$ and $\ini_{\preceq}(h) = y_{1,1}y_{1,3}y_{2,2}y_{2,4}y_{3,3}.$
In particular, $\ini_{\preceq}(h) \in \ini_{\preceq}(I)$ is not multilinear.
However, $\ini_{\preceq}(h)$ is not divisible by any of
$\ini_{\preceq}(p_1)$, $\ini_{\preceq}(p_2)$, and $\ini_{\preceq}(f)$.
Therefore, among the minimal generators of $\ini_{\preceq}(I)$, there exists a non-multilinear monomial.
Thus, the ideal $I$ has no multilinear \groebner basis with respect
to $\preceq$.
In particular, $I$ has no multilinear universal
\groebner basis.
\end{proof}

\begin{lemma}\label{lem:local-obstruction}
  Let $r \geq 3$, and let $Y = (y_{i,j})$ be a $r \times (r+1)$ generic matrix
  with $\deg(y_{i,j}) = \eb_j$.
  Let
  \begin{equation*}
    p_j = \det Y_{[r],[r+1] \setminus \{j\}}.
  \end{equation*}
  An ideal $I$ generated by $k$ distinct maximal minors $p_j$ of $Y$, where
  $2\leq k\leq r-1$, does not have any multilinear universal \groebner basis.
\end{lemma}
\begin{proof}
By permuting columns of $Y$, it suffices to consider
\[
 I=(p_3,p_4,\ldots,p_{k+2}).
\]
We assume that $I$ has a multilinear universal \groebner basis.
By setting some of the variables of $Y$ equal to $0$ and replacing $y_{i-1,i}$ by $t_i$ for $i \geq 5$,
we obtain the matrix

\begin{equation*}
  Y'=
  \begin{pmatrix}
    Z & 0\\
    0 & \diag(t_5,t_6,\ldots,t_{r+1})
  \end{pmatrix},
\end{equation*}
where $Z$ is a $3 \times 4$ generic matrix and $\diag(t_5,t_6,\ldots,t_{r+1})$ is a diagonal matrix.
Let
\begin{equation*}
    q_j = \det Y'_{[r],[r+1] \setminus \{j\}}
  \end{equation*}
  and $J = (q_3, q_4, \ldots, q_{k+2})$. By Lemma~\ref{lem:inheritance}(b), the ideal $J$ has a multilinear universal \groebner basis.
  Then, $q_5, \ldots, q_{k+2}$ are zero and $q_3 = t_5\cdots t_{r+1}h_3$ and
  $q_4 = t_5\cdots t_{r+1}h_4$,
  where $h_j = \det Z_{[3], [4]\setminus \{j\}}$. By setting $t_5 = \ldots = t_{r+1} = 1$, we have that the ideal generated by
  the two minors $h_3$ and $h_4$ of the $3 \times 4$ generic matrix $Z$ has a multilinear universal 
  \groebner basis by Lemma~\ref{lem:inheritance}(c). However, it contradicts Lemma~\ref{lemma:3 times 4}.
\end{proof}

The following proposition and lemma constitute the main part of the
proof of the implication \textup{(ii)}$\Rightarrow$\textup{(iii)} in
Theorem~\ref{thm:main}.

\begin{proposition}\label{prop:degrees}
If a complementary determinantal edge ideal $J_{c}(G)$ has a multilinear universal \groebner basis, then
$d_G(v)\in\{0,1,n-2,n-1\}$ for every $v\in[n]$.
\end{proposition}
\begin{proof}
  For a homogeneous polynomial $f \in J_{c}(G)$, if the $v$th component of the degree of $f$ is $0$,
  the polynomial $f$ belongs to the ideal generated by $\{\mu_{vj}:j\in N_G(v)\}$.
  Therefore, we have
  \begin{equation*}\label{eq:contraction}
  J_{c}(G)\cap \kk[x_{aj}:j\neq v]
  =(\mu_{vj}:j\in N_G(v)).
  \end{equation*}
  The right-hand side is generated by $d_G(v)$ maximal minors
  from the $(n-2)\times(n-1)$ matrix with column set $U_v = [n] \setminus \{v\}$.
  Lemma~\ref{lem:inheritance}(a) and Lemma~\ref{lem:local-obstruction}
  exclude $2\leq d_G(v)\leq n-3$.
\end{proof}

\begin{lemma}\label{lem:graph-classification}
Suppose that $n\geq5$ and $d_G(v)\in\{0,1,n-2,n-1\}$ for all $v$.
Then $G$ has one of the following forms:
\begin{enumerate}[(a)]
\item a matching, together with any isolated vertices;
\item $K_{1,n-1}$, a star with $n-1$ leaves;
\item $K_{1,n-2}+K_1$, consisting of a star with $n-2$ leaves together with one isolated vertex;
\item $K_{n-1}+K_1$, consisting of a complete graph on $[n]\setminus\{v\}$ together with the isolated vertex $v$;
\item a complete graph on $[n]\setminus\{v\}$ with $v$ adjacent to exactly one other vertex;
\item $K_n$ with the edges of a matching removed.
\end{enumerate}
\end{lemma}
\begin{proof}
Set
\[
  V_{\text{high}}
  =\{v\in V(G):d_G(v)\ge n-2\},
  \qquad
  V_{\text{low}}
  =\{v\in V(G):d_G(v)\le1\},
\]
and write
\[
  h=|V_{\text{high}}|,
  \qquad
  \ell=|V_{\text{low}}|.
\]
By the degree condition, these sets form a partition of $V(G)$.
In particular, $h+\ell=n$.

We first show that $h\le1$ or $\ell\le1$.
Suppose that $h,\ell\ge2$.
Let
\[
  A=
    \{(u,v)\in V_{\text{high}}\times V_{\text{low}}:
      \{u,v\}\in E(G)\}
\]
and $a = |A|$.

Each vertex in $V_{\text{low}}$ has degree at most one.
If a vertex $v \in V_{\text{low}}$ is incident to an edge in $A$,
then there exists exactly one vertex
$u \in V_{\text{high}}$ such that $(u,v) \in A$.
Therefore we have
\[
  a
  =\sum_{v\in V_{\text{low}}}
    |N_G(v)\cap V_{\text{high}}|
  \le\ell.
\]
On the other hand, each vertex in $V_{\text{high}}$
has degree at least $n-2$ and has at most $h-1$
neighbors in $V_{\text{high}}$.
It consequently has at least $(n-2)-(h-1)=\ell-1$
neighbors in $V_{\text{low}}$.
Therefore, if a vertex $u \in V_{\text{high}}$ is incident to an edge in $A$,
there are at least $\ell - 1$ vertices $v \in V_{\text{low}}$ such that $(u,v) \in A$.
We have
\[
  a
  =\sum_{u\in V_{\text{high}}}
    |N_G(u)\cap V_{\text{low}}|
  \ge h(\ell-1).
\]
Combining the two inequalities, we obtain $h(\ell-1)\le a\le\ell$, and hence
\begin{equation*}
  (h-1)(\ell-1)\le1.
\end{equation*}
Since $h,\ell\ge2$ are integers, this forces $h=\ell=2$.
Thus $n=h+\ell=4$. It contradicts $n\ge5$.
Therefore, we have $h\le1$ or $\ell\le1$.

If $h=0$,
every vertex has degree at most one, so no two distinct
edges share a vertex.
Hence $G$ is a matching, with isolated vertices allowed.
Therefore, $G$ is as in case~\textup{(a)}.

We consider the case $h=1$.
Let $V_{\text{high}}=\{v\}$.
Then $d_G(v)$ is either $n-1$ or $n-2$.
Every neighbor of $v$ belongs to $V_{\text{low}}$
and already has an incident edge joining it to $v$.
Thus none of these neighbors can be incident to another edge.

If $d_G(v)=n-1$, it follows that $G$ is as in case~\textup{(b)}.
If $d_G(v)=n-2$, let $w$ be the unique vertex other than $v$
that is not adjacent to $v$.
The vertex $w$ cannot be adjacent to any neighbor of $v$,
since each such neighbor already has degree one.
Consequently, $w$ is isolated, and $G$ is as in case~\textup{(c)}.

We consider the case $\ell=1$.
Let $V_{\text{low}}=\{v\}$.

If $v$ is isolated, every vertex in $V_{\text{high}}$ is not adjacent to $v$.
Since a vertex in $V_{\text{high}}$ has at most one non-neighbor among the
other vertices of $G$, it must be adjacent to every other
vertex in $V_{\text{high}}$.
Therefore, $G$ is as in case~\textup{(d)}.

If $d_G(v)=1$, let $w$ be its unique neighbor.
No vertex $u \in V_{\text{high}}$ with $u\neq w$ is adjacent to $v$,
so it must be adjacent to every other vertex in $V_{\text{high}}$,
including $w$.
This also shows that $w$ is adjacent to all other high vertices.
Therefore, $G$ is as in case~\textup{(e)}.

If $\ell=0$,
every vertex has degree at least $n-2$.
Hence
\[
  d_{\overline G}(v)=n-1-d_G(v)\le1
  \qquad\text{for every }v\in V(G).
\]
Therefore, $\overline G$ is a matching, with isolated vertices
allowed, and $G$ is as in case~\textup{(f)}.
\end{proof}

\subsection{Matchings}
In the previous subsection, we identify the candidate graphs for which
the corresponding complementary determinantal edge ideals may have a multilinear universal
\groebner basis.
In this subsection, we complete the proof of the implication \textup{(ii)}$\Rightarrow$\textup{(iii)}
in Theorem~\ref{thm:main}.
We prove that if $n \geq 6$, then complementary determinantal edge ideals of matchings
with at least two edges do not have any multilinear universal \groebner basis.

\begin{lemma}\label{lemma:induced}
  Assume that $J_{c}(G)$ has a multilinear universal \groebner basis for a graph $G$.
  Let $W\subset[n]$ and assume $|W|\geq5$.
  Then, for the induced subgraph $G[W]$, the complementary determinantal edge ideal
  $J_c(G[W])$ has a multilinear universal \groebner basis,
  where $J_c(G[W])$ is formed using a generic $(|W|-2)\times|W|$ matrix.
\end{lemma}
\begin{proof}
  Set $r=n-|W|$. By permuting columns, we may assume that
  $W$ consists of the last $|W|$ vertices of $[n]$ and $[n]\setminus W$ consists of the
  first $r$ vertices of $[n]$. 

By setting some variables of $X$ equal to zero and replacing $x_{i,i}$ by $t_i$ for $i \in [r]$,
we obtain the matrix
\[
  X'=
  \begin{pmatrix}
    \diag(t_1, \ldots, t_r) & 0\\
    0 & Y
  \end{pmatrix},
\]
where
\[
  Y=(x_{i,j})_{\substack{r+1\le i\le n-2\\r+1\le j\le n}}
\]
is a generic $(|W|-2)\times|W|$ matrix.
For distinct $i,j\in W$, put
\[
  \nu_{ij}
  :=\det X_{\{r+1,\ldots,n-2\},\,W\setminus\{i,j\}}.
\]
These are the maximal minors of $Y$ corresponding to the
pairs $\{i,j\}\subset W$.
If $i,j\in W$, the submatrix $X'_{[n-2], [n]\setminus \{i,j\}}$
contains the entire diagonal block, and hence
\[
  \det X'_{[n-2], [n]\setminus \{i,j\}} = t_1\cdots t_r \nu_{ij}.
\]
If at least one of $i,j$ belongs to $[r]$, deleting that
column leaves the corresponding row of the submatrix
identically zero.
Therefore,
\[
  \det X'_{[n-2], [n]\setminus \{i,j\}} =
  \begin{cases}
    t_1\cdots t_r\nu_{ij}, & i,j\in W,\\
    0, & \text{otherwise}.
  \end{cases}
\]

By Lemma~\ref{lem:inheritance}(b), the ideal
\begin{equation*}
  (\det X'_{[n-2], [n]\setminus \{i,j\}} : \{i,j\} \in E(G))
\end{equation*}
has a multilinear universal \groebner basis.

Let 
\begin{equation*}
  J=(\nu_{ij}:\{i,j\}\in E(G[W]))\subset\kk[Y]
\end{equation*}
be the complementary determinantal edge ideal of $G[W]$ formed using $Y$.
For each $i\in[r]$, the variable $t_i$ is the only variable
remaining in the $i$th grading block.
Applying Lemma~\ref{lem:inheritance}(c) successively, we
set $t_1=\cdots=t_r=1$.
The resulting ideal is precisely $J$.
We conclude that $J_c(G[W])$ has a
multilinear universal \groebner basis.
\end{proof}

\begin{lemma}\label{lemma:matching-six}
Let $n = 6$ and let $G_1$ and $G_2$ be matchings with
\begin{equation*}
  E(G_1)=\{\{1,4\},\{2,5\}\},\qquad E(G_2)=\{\{1,4\},\{2,5\},\{3,6\}\}.
\end{equation*}
Then, neither $J_c(G_1)$ nor $J_c(G_2)$ has a multilinear universal \groebner basis.
\end{lemma}
\begin{proof}
Fix $i\in\{1,2\}$ and suppose that
$J_c(G_i)$ has a multilinear universal \groebner basis.
By setting some variables of $X$ equal to zero and replacing $x_{1,1}$ and $x_{1,2}$
by $t_1$ and $t_2$, respectively,
we obtain the matrix
\[
  X'=
  \begin{pmatrix}
    t_1 & t_2 & 0       & 0       & 0       & 0\\
    0   & 0   & x_{2,3} & x_{2,4} & x_{2,5} & x_{2,6}\\
    0   & 0   & x_{3,3} & x_{3,4} & x_{3,5} & x_{3,6}\\
    0   & 0   & x_{4,3} & x_{4,4} & x_{4,5} & x_{4,6}
  \end{pmatrix}.
\]
Let $Y = X'_{\{2,3,4\}, \{3,4,5,6\}}$ be a $3 \times 4$ generic matrix. For $j\in[4]$, put
\[
  q_j=\det Y_{[3],[4]\setminus\{j\}}.
\]
Expanding along the first row gives
\[
  \mu_{14}(X')=t_2q_2,
  \qquad
  \mu_{25}(X')=t_1q_3,
  \qquad
  \mu_{36}(X')=0.
\]
Thus, for either choice of $i$, the image of $J_C(G_i)$
under this specialization is
\[
  J=(t_2q_2,t_1q_3) \subset\kk[Y, t_1, t_2].
\]
By Lemma~\ref{lem:inheritance}(b), the ideal $J$ has a
multilinear universal \groebner basis.

In this grading, $t_1$ and $t_2$ are the only variables
in the first and second grading blocks, respectively.
Applying Lemma~\ref{lem:inheritance}(c) twice, we set $t_1 = t_2 = 1$.
Then, the ideal
\[
  (q_2,q_3)\subset\kk[Y]
\]
has a multilinear universal \groebner basis.
Therefore, the ideal generated by two minors of $Y$ has a multilinear universal \groebner basis.

However, this contradicts Lemma~\ref{lemma:3 times 4}.
Since $i$ was arbitrary, neither $J_c(G_1)$ nor $J_c(G_2)$
has a multilinear universal \groebner basis.
\end{proof}

\begin{corollary}\label{cor:matching-general}
If $n\geq6$ and $G$ is a matching with at least two edges, then
$J_{c}(G)$ does not have any multilinear universal \groebner basis.
In particular, $J_{c}(G)$ is not CS.
\end{corollary}
\begin{proof}
  We assume that $J_{c}(G)$ has a multilinear universal \groebner basis.
If $G$ has exactly two edges, let $W \subset [n]$ be a subset consisting of
their four endpoints and two
isolated vertices. If $G$ has at least three edges, let $W \subset [n]$ be a subset
consisting of
the six endpoints of three of them. In either case the induced six-vertex graph $G[W]$
is, after relabeling, one of the graphs in Lemma~\ref{lemma:matching-six}.
By Lemma~\ref{lemma:induced}, the ideal $J_c(G[W])$ has a multilinear universal \groebner
basis. However, this contradicts Lemma~\ref{lemma:matching-six}.
\end{proof}

\subsection{Cartwright--Sturmfels ideals for the remaining graphs}
In the previous subsection, we identify matchings whose
complementary determinantal edge ideals are not CS
and complete the proof of the implication \textup{(ii)}$\Rightarrow$\textup{(iii)} in
Theorem~\ref{thm:main}.
In this subsection, we then show that the
complementary determinantal edge ideals of all remaining graphs are CS
and prove the implication \textup{(iii)}$\Rightarrow$\textup{(i)}.

\begin{lemma}\label{lem:five-matching-CS}
If $n=5$ and $G$ is a matching consisting of two edges, then $J_{c}(G)$ is CS.
\end{lemma}
\begin{proof}
  It suffices to consider the graph $G$ with the edge set $E(G) = \{\{1,2\}, \{3,4\}\}$.
  Then,
  \begin{equation*}
    J_{c}(G) = (\mu_{12}) + (\mu_{34}).
  \end{equation*}
  The degree of $\mu_{12}$ is $\eb_3+\eb_4+\eb_5$.
  The degree of $\mu_{34}$ is $\eb_1+\eb_2+\eb_5$.
  By Proposition~\ref{proposition:connect one column}, the ideal $J_{c}(G)$ is CS.
\end{proof}

Let $m = n-2$ and let
\[
 J=I_{m}(X),\qquad J_v=I_{m}(X_{[m],U_v})
\]
be the determinantal ideals generated by the maximal minors of
$X$ and $X_{[m],U_v}$, respectively.
We use the fact, proved in~\cite{ConcaDeNegriGorlaUGBforMM}, that $J$ and $J_v$ are CS.
Let $\mathfrak c_v=(x_{1v},\ldots,x_{mv})$.
The ideals $\mathfrak c_v$, $\mathfrak c_i+\mathfrak c_j$, and finite
intersections of such ideals satisfy the hypotheses of
Lemma~\ref{lem:filter}.

For $U_v=\{u_1<\cdots<u_{m+1}\}$ and $a\in[m]$, one has
\begin{equation*}\label{eq:EN}
 \sum_{k=1}^{m+1}(-1)^{k-1}x_{a,u_k}\mu_{v,u_k}=0.
\end{equation*}
Indeed, if we prepend row $a$ of $X_{[m],U_v}$ to that matrix and expand the
resulting determinant along the first row, we obtain $0$.
If $N_{\overline{G}}(v)=\{w\}$, where $\overline{G}$ is the complement graph of $G$,
all terms except the one indexed
by $w$ involve generating minors of $J_{c}(G)$. Consequently,
\begin{equation}\label{eq:one-missing-minor}
 x_{aw}\mu_{vw}\in J_{c}(G)\qquad(a\in[m]).
\end{equation}

\begin{proposition}\label{prop:positive}
Each graph in Lemma~\ref{lem:graph-classification}(b)--(f) has a CS complementary determinantal edge
ideal.
\end{proposition}
\begin{proof}
\noindent\textbf{(b)}:
If $G$ is a complete star $K_{1, n-1}$ centered at $v$, then $J_{c}(G)=J_v$,
which is CS.

\smallskip
\noindent\textbf{(c)}:
Suppose that $G$ is a star $K_{1,n-2}$ centered at $v$,
and let $w$ be the unique vertex that is not adjacent to $v$.
Then $J_{c}(G)=(\mu_{vj}:j\in U_v\setminus\{w\}).$

Consider the homogeneous decomposition $J_v=\bigoplus_{\ab \in\mathbb Z_{\ge0}^n}(J_v)_\ab.$
Since $\mathfrak c_w$ is generated by all variables in column $w$,
\[
  (\mathfrak c_w)_\ab=
  \begin{cases}
    S_\ab, & a_w>0,\\
    0, & a_w=0,
  \end{cases}
\]
where $a_w$ is the $w$th component of $\ab$.
Consequently,
\[
  J_v\cap\mathfrak c_w
  =
  \bigoplus_{\substack{\ab\in\mathbb Z_{\ge0}^n\\
                       a_w>0}}
    (J_v)_\ab.
\]
Thus $J_v\cap\mathfrak c_w$ is obtained by retaining precisely
the components of $J_v$ whose degree in column $w$ is positive.

We claim that
\begin{equation}\label{eq:almost-star}
  J_{c}(G)=J_v\cap\mathfrak c_w.
\end{equation}
For every $j\in U_v\setminus\{w\}$, the minor $\mu_{vj}$
uses column $w$, so it has degree one in the $w$th column.
Hence every generator of $J_{c}(G)$ belongs to
$J_v\cap\mathfrak c_w$.

For the reverse inclusion, it suffices to take a homogeneous
$f\in(J_v)_\ab$ with $a_w>0$.
Since $J_v$ is generated by the homogeneous polynomials
$\mu_{vj}$, we may write
\[
  f=
  \sum_{j\in U_v\setminus\{w\}}h_j\mu_{vj}
  +h_w\mu_{vw},
\]
where each nonzero $h_j$ is homogeneous of degree
$\ab-\deg(\mu_{vj})$.
The terms in the first sum belong to $J_{c}(G)$.
Since $\mu_{vw}$ does not use column $w$, the
$w$th component of $\deg(h_w)$ is positive.
Thus $h_w\in\mathfrak c_w$.

At the vertex $v$, the only missing neighbor is $w$.
By~\eqref{eq:one-missing-minor}, we have
\[
  x_{aw}\mu_{vw}\in J_{c}(G)
  \qquad(a\in[m]),
\]
and hence
\[
  \mathfrak c_w\mu_{vw}\subset J_{c}(G).
\]
It follows that $h_w\mu_{vw}\in J_{c}(G)$, and therefore $f\in J_{c}(G)$.
This proves~\eqref{eq:almost-star}.
Since $J_v$ is CS, Lemma~\ref{lem:filter} shows that $J_{c}(G)$ is CS.

\smallskip
\noindent\textbf{(d)}:
Suppose that $G[U_v]$ is complete and $v$ is isolated.
Using the same argument as in the previous case, we prove
\begin{equation}\label{eq:isolated}
  J_{c}(G)=J\cap\mathfrak c_v
  =
  \bigoplus_{\substack{\ab\in\mathbb Z_{\ge0}^n\\
                       a_v>0}}
    J_\ab.
\end{equation}
Indeed, every generator $\mu_{ij}$ of $J_{c}(G)$ has $i,j\neq v$,
so it uses column $v$.
This proves $J_{c}(G)\subset J\cap\mathfrak c_v$.

Conversely, let $f\in J_\ab$ be homogeneous with
$a_v>0$.
Express $f$ as a sum of homogeneous multiples of the maximal
minors of $X$.
The minors not belonging to the generating set of $J_{c}(G)$
are precisely $\mu_{vj}$ with $j\in U_v$.
Since each $\mu_{vj}$ has degree zero in column $v$,
its coefficient in this expression belongs to $\mathfrak c_v$.

For every $j\in U_v$, the only missing neighbor of $j$ is $v$.
Thus~\eqref{eq:one-missing-minor} gives
\[
  x_{av}\mu_{vj}\in J_{c}(G)
  \qquad(a\in[m],\ j\in U_v).
\]
Therefore, $f$ belongs to $J_{c}(G)$,
proving~\eqref{eq:isolated}.
Lemma~\ref{lem:filter}, applied to $J$ and $\mathfrak c_v$,
proves that $J_{c}(G)$ is CS.

\smallskip
\noindent\textbf{(e)}:
Suppose that $G[U_v]$ is complete and $N_G(v)=\{w\}$.
Adding the edge $vw$ to the graph in the preceding case gives
\begin{equation}\label{eq:pendant}
  J_{c}(G)=(J\cap\mathfrak c_v)+(\mu_{vw}).
\end{equation}
Since $\mu_{vw}\in J$ and $\mu_{vw}$ does not use column $v$,
we obtain
\[
  (J_{c}(G))_\ab=
  \begin{cases}
    J_\ab, & a_v>0,\\
    ((\mu_{vw}))_\ab, & a_v=0.
  \end{cases}
\]

Every coordinate change $\gamma \in \calG$ preserves
the grading.
Hence
\[
  (\gamma J_{c}(G))_\ab=
  \begin{cases}
    (\gamma J)_\ab, & a_v>0,\\
    ((\gamma\mu_{vw}))_\ab, & a_v=0.
  \end{cases}
\]
Choosing $\gamma$ sufficiently general,
we obtain
\begin{equation}\label{eq:pendant-gin}
  \gin_{\preceq}(J_{c}(G))
  =
  \bigl(\gin_{\preceq}(J)\cap\mathfrak c_v\bigr)
  +\bigl(\ini_{\preceq}(\gamma\mu_{vw})\bigr).
\end{equation}

Since $J$ is CS, the first summand is a squarefree monomial ideal.
Moreover, $\gamma\mu_{vw}$ has degree
$\one-\eb_v-\eb_w$, so its initial monomial is a squarefree monomial.
Thus the right-hand side of~\eqref{eq:pendant-gin} is generated
by squarefree monomials.
Therefore, $J_{c}(G)$ is CS.

\smallskip
\noindent\textbf{(f)}:
Let $G=K_n\setminus M$, where $M$ is a matching, and set
\[
  Q_M=\bigcap_{\{i,j\}\in M}(\mathfrak c_i+\mathfrak c_j),
\]
with the empty intersection interpreted as $S$.

For each $\{i,j\}\in M$, we have
\[
  (\mathfrak c_i+\mathfrak c_j)_\ab=
  \begin{cases}
    S_\ab, & a_i+a_j>0,\\
    0, & a_i=a_j=0.
  \end{cases}
\]
Consequently,
\[
  (Q_M)_\ab=
  \begin{cases}
    S_\ab,
      & a_i+a_j>0
        \text{ for every }\{i,j\}\in M,\\
    0, & \text{otherwise},
  \end{cases}
\]
and therefore
\[
  J\cap Q_M
  =
  \bigoplus_{\substack{\ab\in\mathbb Z_{\ge0}^n\\
    a_i+a_j>0\ \text{for every }\{i,j\}\in M}}
    J_\ab.
\]
Thus $J\cap Q_M$ retains precisely those components of $J$
whose degrees are positive on at least one endpoint of
every edge in $M$.

We claim that
\begin{equation}\label{eq:deleted-matching}
  J_{c}(G)=J\cap Q_M.
\end{equation}
First, let $\mu_{ab}$ be a generator of $J_{c}(G)$, and fix a
edge $\{i,j\}\in M$.
Since $\{a,b\}\neq\{i,j\}$, the column set
$[n]\setminus\{a,b\}$ contains at least one of $i,j$.
Therefore, we have $\mu_{ab}\in\mathfrak c_i+\mathfrak c_j$ for every $\{i,j\} \in M$
and we have $J_{c}(G)\subset J\cap Q_M.$

For the reverse inclusion, we first consider a edge $\{i,j\}\in M$.
Since $M$ is a matching, $j$ is the only non-neighbor of $i$,
and $i$ is the only non-neighbor of $j$.
Applying~\eqref{eq:one-missing-minor} at these two vertices yields
\[
  x_{ai}\mu_{ij}\in J_{c}(G),
  \qquad
  x_{aj}\mu_{ij}\in J_{c}(G)
  \qquad(a\in[m]).
\]
It follows that
\[
  (\mathfrak c_i+\mathfrak c_j)\mu_{ij}\subset J_{c}(G).
\]

Now take a nonzero homogeneous polynomial
$f\in(J\cap Q_M)_\ab$.
By the description of $(Q_M)_\ab$, its degree satisfies
$a_i+a_j>0$ for every $\{i,j\}\in M.$
Since $f\in J$, we may write
\[
  f=
  \sum_{\{a,b\}\in E(G)}h_{ab}\mu_{ab}
  +
  \sum_{\{i,j\}\in M}h_{ij}\mu_{ij},
\]
choosing each nonzero coefficient to be homogeneous so that
every nonzero summand has degree $\ab$.
All summands in the first sum belong to $J_{c}(G)$.

We consider the summands in the second sum. For each $\{i,j\} \in M$,
the minor $\mu_{ij}$ has degree zero in both columns $i$ and $j$.
Thus the $i$th and $j$th components of $\deg(h_{ij})$
are $a_i$ and $a_j$, respectively.
Since $a_i + a_j > 0$,  we have $h_{ij}\in\mathfrak c_i+\mathfrak c_j$.
Thus, we have
\[
  h_{ij}\mu_{ij}\in (\mathfrak c_i+\mathfrak c_j)\mu_{ij}\subset J_{c}(G).
\]
Therefore, every summand belongs to $J_{c}(G)$, proving
$f\in J_{c}(G)$ and establishing~\eqref{eq:deleted-matching}.

Since $Q_M$ satisfies the hypotheses of Lemma~\ref{lem:filter},
and since $J$ is CS, $J_{c}(G)$ is CS.
\end{proof}

\begin{proof}[Proof of Theorem~\ref{thm:main}]
The implication (i)$\Rightarrow$(ii) follows from
Proposition~\ref{proposition:generic initial ideal of Cartwright--Sturmfels ideal}.
Proposition~\ref{prop:degrees}, Lemma~\ref{lem:graph-classification} and
Corollary~\ref{cor:matching-general} give (ii)$\Rightarrow$(iii).
Finally, Lemma~\ref{lem:five-matching-CS} and Proposition~\ref{prop:positive}
give (iii)$\Rightarrow$(i).
\end{proof}

\begin{corollary}
  If $G$ satisfies condition (iii) of Theorem~\ref{thm:main},
  then $J_{c}(G)$ is radical.
\end{corollary}
\begin{proof}
  It follows from Theorem~\ref{thm:main} and Proposition~\ref{proposition:generic initial ideal of Cartwright--Sturmfels ideal}.
\end{proof}

\subsection*{Acknowledgements}
The author would like to thank Akihiro Higashitani for helpful discussions.
The author gratefully acknowledges the assistance of generative AI,
particularly ChatGPT (OpenAI), during the exploratory and editorial stages
of this work. All mathematical
statements, proofs, and references remain the sole responsibility of the
author. The author is indebted to the developers of the Julia package
Oscar~\cite{OSCAR}, which enabled explicit computations that were crucial
in discovering our results. The author gratefully acknowledges support
from JST SPRING, Grant Number JPMJSP2138.

\bibliography{references}
\bibliographystyle{amsplain}
\end{document}